\documentclass{amsart}

\usepackage{macros}
\standardmargins\standardpackages\standardrenews\standardlabeling\standardmakes
\colorcommentstrue

\let\comjohannes\undefined
\allowcomments{\comjohannes}{JS}{Johannes}{magenta}

\draftfalse

\authormumtaz
\authorjohannes
\authordavid

\newcommand{\WWW}{W_n^{\theta}(\Psi)}
\newcommand{\WWs}{W_n^{\theta}(\psi)}
\newcommand{\WWh}{W_n^0(\psi)}
\newtheorem*{thmsch}{Theorem SBV}
\newtheorem*{GBSP}{Generalised Baker--Schmidt Problem for Hausdorff Measure}

\newcommand{\SSS}{S_n^{\bftheta}(\Psi)}

\usepackage[pagewise]{lineno}

\begin{document}
\title{The generalised Baker--Schmidt problem on hypersurfaces}

\begin{abstract}
The Generalised Baker--Schmidt Problem (1970)  concerns the $f$-dimensional Hausdorff measure of the set of $\psi$-approximable points on a nondegenerate manifold. There are two variants of this problem, concerning simultaneous and dual approximation. Beresnevich--Dickinson--Velani (in 2006, for the homogeneous setting) and Badziahin--Beresnevich--Velani (in 2013,  for the inhomogeneous setting) proved the divergence part of this problem for dual approximation on arbitrary nondegenerate manifolds. The corresponding convergence counterpart represents a major challenging open question and the progress thus far has only been attained over planar curves. In this paper, we settle this problem for  hypersurfaces  in a more general setting, i.e. for inhomogeneous approximations and with a non-monotonic multivariable approximating function.
\end{abstract}
\maketitle

\section{Introduction: Diophantine approximation on manifolds}
\label{sek1}

Let $n\geq 1$ be a fixed integer.
Let $\Psi:\Z^n\to[0, \infty)$ be a \emph{multivariable approximating function}, i.e. $\Psi$ has the property that
$\Psi(\qq) \rightarrow~0 \text{ as } \|\qq\|:=\max(|q_1|, \ldots, |q_n|) \rightarrow~\infty.$ Further, let $\theta:\R^n\to \R$
be a continuous function. We consider the dual approximation 
problem with respect to the approximation function $\Psi$
and the functional inhomogeneous parameter $\theta$. Concretely, we are concerned with the  set

\begin{equation*}
  \WWW:=\left\{\xx=(x_1,\dots,x_n)\in\R^n:\begin{array}{l}
  |q_1x_1+\cdots+q_nx_n-p-\theta(\xx)|<\Psi(\qq) \\[1ex]
  \text{for} \ \ i.m. \ (p, q_1, \ldots, q_n)\in\Z^{n+1}\}
                           \end{array}
\right\},
\end{equation*}
where  $`i.m.$' stands for `infinitely many'.  A vector $\xx \in \R^n$ will be called \emph{$(\Psi, \theta)$-approximable} if it lies in the set $\WWW$. In the case where the function $\theta$ is constant, the set $\WWW$ corresponds to the familiar inhomogeneous setting and in particular, when $\theta= 0$ the problem reduces to the \emph{homogeneous} setting. We are interested in the `size' of the set $\WWW$ with respect to the  $f$-dimensional Hausdorff measure $\HH^f$ for some dimension function $f$, i.e.
an increasing continuous function $f:\mathbb{R}\to \mathbb{R}$ with $f(0)=0$.
We refer to Subsection \ref{hm} below for a brief introduction to Hausdorff measure.

The most modern result in regards to determining the size of the set $\WWW$ is the following statement due to the contributions of many authors but most importantly due to the works of Schmidt \cite{Schmidt8} and Beresnevich \& Velani \cite{BeresnevichVelani4}.

\begin{thmsch}{\em
Let $\Psi$ be a multivariable approximating function. Let $f$ be a dimension function such that $r^{-n}f(r)\to\infty$ as $r\to 0.$ Assume that $r\mapsto r^{-n}f(r)$ is decreasing and $r\mapsto r^{1-n}f(r)$ is increasing. Fix $\theta\in\R$. Then
\begin{equation*}
  \HH^f( \WWW)=\left\{\begin{array}{cl}
 0 &  {\rm if } \quad\sum\limits_{\qq\in\Z^n\setminus \{\0\}}\|\qq\|^n\Psi(\qq)^{1-n}f\left(\frac{\Psi(\qq)}{\|\qq\|}\right)< \infty.\\[3ex]
 \infty &  {\rm if } \quad \sum\limits_{\qq\in\Z^n\setminus \{\0\}}\|\qq\|^n\Psi(\qq)^{1-n}f\left(\frac{\Psi(\qq)}{\|\qq\|}\right)=\infty \text{ and $\Psi$ is decreasing or $n\geq 2$}.
                                     \end{array}\right.
    \end{equation*}}
\end{thmsch}

 Note that $\HH^f$ is proportional to the standard Lebesgue measure when $f(r)=r^n$. The convergence case can be settled rather straightforwardly by using the Hausdorff--Cantelli lemma. 
The divergence case carries the main substance, and that is where the condition on the approximating function comes into play. Theorem SBV is very significant as it encompasses many classical landmark results such as Khintchine's theorem, Jarn\'ik's theorem, and Groshev's theorem.

The problem of estimating the size of  the set $\WWW$ becomes more intricate if one restricts $\xx\in\R^n$ to lie on a $k$-dimensional, nondegenerate\footnote{In this context `nondegenerate' means suitably curved, see \cite{Beresnevich3,KleinbockMargulis2} for precise formulations. }, analytic submanifold $\MM\subseteq \R^n$. In this branch of the theory, such a restriction essentially presumes that the components of each point of interest $\xx$ are functionally related. For this reason, the description of problems of this type is often referred to as the theory of \emph{approximation of dependent quantities}. When asking such questions it is natural to phrase them in terms of a suitable measure supported on the manifold, since when $k<n$ the $n$-dimensional Lebesgue measure of $\MM\cap \WWW$ is zero irrespective of $\Psi$. For this reason,  results in the dependent Lebesgue theory (for example, Khintchine--Groshev type theorems for manifolds) are posed in terms of the $k$-dimensional 
Lebesgue measure (=Hausdorff measure) on $\MM$.

In full generality, a complete Hausdorff measure treatment akin to Theorem SBV for manifolds $\MM$  represents a deep open problem in the theory of Diophantine approximation. The problem is referred to as the Generalised Baker-Schmidt Problem (GBSP) inspired by the pioneering work of Baker $\&$ Schmidt \cite{BakerSchmidt}.  Ideally one would want to solve the following problem in full generality.

\begin{GBSP} {\em Let $\MM$ be a nondegenerate submanifold of $\R^n$ with $\dim\MM=k$ and $n \geq 2$. Let $\Psi$ be a multivariable approximating function.  Let $f$ be a dimension function such that $r^{-k}f(r)\to\infty$ as $r\to 0.$ Assume that $r\mapsto r^{-k}f(r)$ is decreasing and $r\mapsto r^{1-k}f(r)$ is increasing. Prove that }

\begin{equation*}
  \HH^f( \WWW\cap\MM)=\left\{\begin{array}{cl}
 0 &  {\rm if } \quad\sum\limits_{\qq\in\Z^n\setminus \{\0\}}\|\qq\|^k\Psi(\qq)^{1-k }f\left(\frac{\Psi(\qq)}{\|\qq\|}\right)< \infty.\\[3ex]
 \infty &  {\rm if } \quad \sum\limits_{\qq\in\Z^n\setminus \{\0\}}\|\qq\|^k\Psi(\qq)^{1-k}f\left(\frac{\Psi(\qq)}{\|\qq\|}\right)=\infty.
                                     \end{array}\right.
\end{equation*}
\end{GBSP}

In fact, this problem is stated in the most idealistic format and solving it in this form is extremely challenging. The main difficulties lie
 in the convergence case and therein constructing a suitable nice cover for the set $ \WWW\cap\MM$. We list the contributions to date to highlight the significance of this problem.

\smallskip

\noindent{\bf Notation.} In the case where the dimension function is of the form $f(r):=r^s$ for some $s < k$, $\HH^f$ is simply denoted as $\HH^s$. Sometimes we will consider functions of the form $\Psi(\qq)=\psi(\|\qq\|)$, and in this case we use $W_n^\theta(\psi)$ as a shorthand for $W_n^\theta(\Psi)$. The function $\psi:\R_{>0} \to \R_{>0}$ here is called a \emph{single-variable approximating function}. The Veronese curve is denoted as $V_n:=\{(x,x^{2},\ldots,x^n):x\in\R\}$.
  By $B_n(\xx, r)$ we mean the ball centred at the point $\xx\in\mathbb{R}^n$ of radius $r$.
For real quantities $A,B$ and a parameter $t$, we write $A \lesssim_t B$ if $A \leq c(t) B$ for a constant $c(t) > 0$ that depends on $t$ only (while $A$ and $B$ may depend on other parameters). 
We write  $A\asymp_{t} B$ if $A\lesssim_{t} B\lesssim_{t} A$.
If the constant $c>0$ depends only on parameters that are constant throughout a proof, we simply write $A\lesssim B$ and $B\asymp A$.
\subsection{Historical progression towards GBSP for Hausdorff measure}

We list first the progress for the Lebesgue measure theory:

\begin{itemize}
\item[1932.] Diophantine approximation on manifolds dates back to the profound conjecture of K. Mahler \cite{Mahler2}, which can be rephrased as the statement that $\HH^1(\WWh\cap V_n)=0$ for $\psi(r)=r^{-\tau}$ with $\tau>n$. 
\item[1965.] Mahler's conjecture was eventually proven in 1965 by Sprind\v zuk \cite{Sprindzuk}, who then conjectured that the same holds when $V_n$ is replaced by any nondegenerate analytic manifold, and $\HH^1$ is replaced by $\HH^k$ where $k$ is the dimension of this manifold. 

\item[1998.] Although particular cases of Sprind\v zuk's conjecture were known, it was not until 1998 that Kleinbock \&
   Margulis \cite{KleinbockMargulis2} established Sprind\v zuk's conjecture in full generality by using dynamical tools based on diagonal flows on homogeneous spaces. This breakthrough result acted as a catalyst for the subsequent progress in this area of research.
   
   \item [1999--2006.] The convergence Lebesgue measure result for $\WWh\cap\MM$ was established independently in \cite{Beresnevich3} and \cite{BKM2} and the divergence case was established in \cite{BBKM}. 
  
   \item [2013.] The first inhomogeneous Lebesgue measure result for the divergence case of $\WWW\cap\MM$ was established very recently in \cite{BBV}.

\end{itemize}

  With regards to the Hausdorff measure or dimension theory for dual approximation on manifolds, 
  the progress has proven difficult.  

\begin{itemize}
\item [1970.] The first major result in this direction relating to approximation of points on the Veronese curve appeared in the landmark paper of Baker $\&$ Schmidt \cite{BakerSchmidt} in which they proved upper and lower bounds for Hausdorff dimension on the Veronese curve. Further, they conjectured that their lower bound is actually sharp,
which was later proven to be true in \cite{Bernik}.
\item [2000.] Dickinson and Dodson \cite{DickinsonDodson2} proved a lower bound for Hausdorff dimension on extremal manifolds for approximating functions 
of the type $\psi(r)=r^{-\tau}$ for any $\tau>n$.
\item [2006.] A lower bound on the $\HH^f$ measure of $\WWh\cap\MM$ was established in \cite{BDV} as a consequence of their ubiquity framework.
\item [2013.] The first ever inhomogeneous result regarding the Hausdorff measure ($\HH^f$ measure) of $\WWW\cap\MM$ for divergence was established in \cite{BBV} but only for the dimension function $f(r)=r^s$ and with a certain convexity condition on the multivariable approximating function $\Psi$. We note that their proof can be extended to a general dimension function $f$ without any hassles. 
\item [2015.] Proving the genuine Hausdorff measure result for convergence for any manifold remained out of reach until recently when the first-named author proved it  in \cite{Hussain} for $\WWh\cap V_2$ i.e. for $\HH^f$ measure on the parabola, for single-variable approximating functions, under some mild assumptions on the dimension function $f$.
\item [2017.] Huang proved \cite{Huang}  the homogeneous convergence result over planar curves, showing that the $\HH^s(\WWs\cap\CC)=0$ for all non-degenerate planar curves $\CC$ if a certain sum converges. Soon after that, Badziahin--Harrap--Hussain \cite{BHH} extended Huang's result to the inhomogeneous setting but still within the 
framework of a single-variable approximating function $\psi$. 

\item [2019.] Very recently, authors of this paper have proved several results regarding the $\HH^f$-measure on non-degenerate planar curves  and Veronese curves in any dimension \cite{HSS2}. In particular, for the $\HH^f$-measure on the parabola,  it is proved that the monotonicity assumption on the multivariable approximating function cannot be removed.  For the single-variable approximating function, along with many other results which improve our understanding of the GBSP on curves, the GBSP for Veronese curves in any dimension $n$ is also discussed. Concretely, we obtain a generalisation of a recent result of Pezzoni \cite{Pezzoni}.
\end{itemize}

The main point of this discussion is that beyond the planar curves hardly anything is known for the convergence case for $\HH^f$-measure for both homogeneous and inhomogeneous settings. We prove the GBSP for Hausdorff measure over hypersurfaces beyond the setting of planar curves, with the additional perk that we can handle non-monotonic multivariable approximating functions.

\begin{remark}
An analogue for simultaneous approximation of the GBSP for Hausdorff measure can easily be formulated for the set
the set 
\begin{equation*}
  \SSS:=\left\{\xx=(x_1,\dots,x_n)\in\R^n: \begin{array}{l}\max\limits_{1\leq i\leq n} 
  |qx_i-p_i-\theta_i(x)|<\Psi(q)\\
  \text{for \emph{i.m.} }\ (p_1, \ldots, p_n, q)\in\Z^n\times \N
                           \end{array}
\right\},
\end{equation*}
and in this regard similar breakthroughs to those outlined above have been made.  Building on results
from \cite{BDV}, a complete homogeneous treatment (for both Lebesgue and Hausdorff measure) of the divergence part of the simultaneous GBSP very recently appeared in Beresnevich's seminal
paper \cite{Beresnevich_Khinchin}. In the convergence case, the most modern results deal with approximation
on planar curves in the homogeneous setting \cite{VaughanVelani} and in the inhomogeneous setting \cite{BVV} for both Lebesgue and Hausdorff measure and for multiple approximating functions \cite{HussainYusupova}. To reiterate, establishing the convergence part of the simultaneous GBSP for arbitrary nondegenerate analytic manifolds remains an open problem. However, very recently, Huang established the GBSP for Hausdorff measure for hypersurfaces \cite{Huang2} in the homogeneous setting by using new estimates on counting the rational points close to the hypersurfaces.
We point out that, in contrast
to the dual approximation, a general formula for the Hausdorff dimension of the set $\SSS\cap\MM$
cannot be expected even in the setting of homogeneous approximation ($\bftheta=\0$) on smooth planar curves, as
soon as the approximation function $\Psi$ decays
sufficiently fast. Indeed, the set of rational
points on the curve plays a substantial role, and  it depends in a sensitive way on the algebraic nature 
of the curve
and might even be empty. This effect has
been illustrated in \cite[Section~2.2]{Beresnevich_Khinchin} and   for an illustration
of this effect for different planar curves, see also \cite{Drutu}
 for 
rational quadrics and \cite{Schleischitz4} for 
algebraic varieties.

\end{remark}
\subsection{Statements and corollaries of main results} We first state some regularity conditions on the manifold $\MM$ and
the dimension function $f$.

\begin{itemize}
\item[(I)] Let $f$ be a dimension function satisfying
 \begin{equation}
\label{fhypothesis}
f(xy) \lesssim x^s f(y) \text{ for all } y < 1 < x
\end{equation}
for some $s < 2(n-2)$.
\end{itemize}
Obviously, (I) is satisfied for the dimension function $f(x) = x^s$ whenever $s < 2(n-2)$. In particular, if $n\geq 3$ then this is true for all $s < n-1$.
Moreover, upon the standard condition on an arbitrary dimension
function $f$ in the GBSP that $f(q)/q^{k}$ (where $k=\dim \MM$) is
decreasing, as soon as $k<2(n-2)$ the condition (I) is again satisfied. To see this
observe that $x>1$ implies $xy>y$ and thus
\[
\frac{f(xy)}{(xy)^{k}}\leq \frac{f(y)}{y^{k}} \quad \Longleftrightarrow \quad
x^{k}f(y)\geq f(xy), 
\]
so we may let $s=k$ if $k<2(n-2)$
(the argument just happens to fail slightly when $k=2(n-2)$).

From now on we will assume the manifold $\MM$ to be a hypersurface. For convenience, we will furthermore assume that $\MM$ is the graph of a smooth map $g:U\to \R$, where $U\subset\R^{n-1}$ is a connected bounded open set. We denote this as $\MM = \Gamma(g)$.

\begin{itemize}
\item[(II)]  Assume that the Hessian $\nabla^2 g$ of the map $g$ satisfies
 \begin{equation}
\label{Hessian}
\HH^f\left(S_\MM \df \{\xx\in U : \nabla^2 g(\xx) \text{ is singular}\}\right)=0.
\end{equation}
\end{itemize}

We discuss the strengths and weaknesses of this condition in the section \ref{sectionfibering}.

\begin{remark}
The condition \eqref{Hessian} can be rewritten in a coordinate-free way using the concept of the \emph{second fundamental form} of a hypersurface. By definition, if $\MM$ is the graph of a function $g$ then the second fundamental form of $\MM$ is the map that sends a point $\yy = (\xx,g(\xx))$ to the bilinear form $$\mathrm{II}_\yy~:~T_\yy \MM \times T_\yy \MM \to \R$$ defined by the formula
\[
\mathrm{II}_\yy\big((\vv,g'(\xx)[\vv]),(\ww,g'(\xx)[\ww])\big) = g''[\vv,\ww].
\]
It can be shown that up to a constant of proportionality, the second fundamental form is independent of the chosen coordinate system (with respect to linear changes of coordinates). Thus, \eqref{Hessian} is equivalent to the coordinate-free condition
\[
\HH^f \big( \{\yy\in\MM : \mathrm{II}_\yy \text{ is singular}\} \big) = 0.
\]
Note that this condition makes sense even if $\MM$ cannot be globally written as the graph of a function, since for each $\yy\in\MM$, $\mathrm{II}_\yy$ can be defined using a coordinate system such that $\MM$ is locally the graph of a function. 
\end{remark}

\begin{theorem}\label{thm1} Let $\Psi$ be a multivariable approximating function. Let $f$ be a dimension function satisfying \text{(I)} and let $g$ be a $C^2$ function satisfying \text{(II)}. Then if $\MM$ is the graph of $g$ and $\theta:\MM\to \R$ is $C^2$, then
\begin{equation}\label{fw}
\HH^f(\WWW\cap\MM)=0
\end{equation}
if the series
\begin{equation}\label{eqcon}
\sum\limits_{\qq\in\Z^n\setminus \{\0\}}\|\qq\|^{n-1}\Psi(\qq)^{2-n}f\left(\frac{\Psi(\qq)}{\|\qq\|}\right)
\end{equation}
converges.

\end{theorem}

Since some
regularity condition as in (II) is inevitable, by our comment
on condition (I) our solution of GBSP for
hypersurfaces can be considered complete as soon as the space 
dimension is at least $n\geq 4$ as then $k=n-1<2(n-2)$.

As stated earlier, the divergence counterpart of our theorem was proved by Badziahin--Beresnevich--Velani \cite{BBV} but for the $s$-dimensional Hausdorff measure and for the multivariable approximating function satisfying the property $\bfP$.  Here, by following the terminology of \cite{BBV}, we say that an approximating function $\Psi$ satisfies \emph{property $\bfP$} if it is of the form
$\Psi(\qq)=\psi(\|\qq\|_{\vv})$ for a monotonically decreasing
function $\psi: \mathbb{R}_{>0}\to \mathbb{R}_{>0}$ 
(single-variable approximation function), 
$\vv=(v_{1},\ldots,v_n)$ with $v_i > 0$ and $\sum_{1\leq i\leq n} v_{i}=n$,
and $\| \cdot \|_{\vv}$ defined as
the quasi-norm $\|\qq\|_{\vv}= \max_i \vert q_{i}\vert^{1/v_{i}}$.

It is to be noted that the techniques used in proving \cite[Theorem 2]{BBV}  also work for the  $f$-dimensional Hausdorff measure situation. Combining our result with \cite[Theorem~2]{BBV}, we obtain a zero-infinity law for the $f$-dimensional Hausdorff measure for the dual inhomogeneous approximation problem on hypersurfaces. 
\begin{corollary}
Let $\theta:\MM\to \R$ be a $C^2$ function, and let $\Psi$
be a multivariable approximating
function satisfying property $\bfP$. Let $f$ be a dimension function satisfying \text{(I)} and let $g$ be a $C^2$ function satisfying \text{(II)}. Then
\begin{equation*}
  \HH^f( \WWW\cap \MM)= \begin{cases}
0 &  {\rm if } \quad \sum\limits_{\qq\in\Z^n\setminus \{\0\}}\|\qq\|^{n-1}\Psi(\qq)^{2-n}f\left(\frac{\Psi(\qq)}{\|\qq\|}\right)< \infty,\\[3ex]
\infty &  {\rm if } \quad \sum\limits_{\qq\in\Z^n\setminus \{\0\}}\|\qq\|^{n-1}\Psi(\qq)^{2-n}f\left(\frac{\Psi(\qq)}{\|\qq\|}\right)=\infty,
\end{cases}
\end{equation*}
where $\MM$ is the graph of $g$.
\end{corollary}
Note that property (II) implies that the graph of $g$ is $2$-nondegenerate in the sense of \cite{BBV}, which allows us to apply \cite[Theorem 2]{BBV}.

 We emphasise that, as a consequence of the property $\bfP$, the divergence case assumes monotonicity assumption on the approximating function $\Psi$. However, the convergence case does not require any monotonicity assumption on the approximating function. Considering the case $f(r) = r^s$ yields the following:

\begin{corollary}\label{cor2}
Let $\theta:\MM\to \R$ be a $C^2$ function, and let $\Psi$
be a multivariable approximating
function satisfying property $\bfP$. Fix $s<2(n-2)$, and let $g$ be a $C^2$ function satisfying property \text{(II)} with $f(r) = r^s$. Then
\begin{equation*}
  \HH^s( \WWW\cap \MM)= \begin{cases}
0 &  {\rm if } \quad \sum\limits_{\qq\in\Z^n\setminus \{\0\}}\|\qq\|\left(\frac{\Psi(\qq)}{\|\qq\|}\right)^{s+2-n}< \infty,\\[3ex]
\infty &  {\rm if } \quad \sum\limits_{\qq\in\Z^n\setminus \{\0\}}\|\qq\|\left(\frac{\Psi(\qq)}{\|\qq\|}\right)^{s+2-n}=\infty,
\end{cases}
\end{equation*}
where $\MM$ is the graph of $g$.
\end{corollary}

\medskip

There are many benefits of a characterisation of $\WWW\cap \MM$ in terms of Hausdorff measure; one such benefit is that such a characterisation implies a formula for the Hausdorff dimension. In particular, let $\tau_{\Psi}$ be the lower order at infinity of $1/\Psi$, that is,
\[
\tau_\Psi:=\liminf_{t\to\infty} 
\frac{\log(1/\Psi(t))}{\log t},
\qquad \text{where} \quad \Psi(t)= \inf_{\xx\in\R^n: \|\xx\|=t} \Psi(\xx).
\]
Then Theorem \ref{thm1}  implies that for any approximating function $\Psi$ with lower order at infinity $\tau_{\Psi}$, we have
\[
\dim_\HH (\WWW\cap\MM)\leq \frac{\tau_{\Psi}(n-2)+2n-1}{\tau_{\Psi}+1}=n-2+\frac{n+1}{\tau_{\Psi}+1}.
\]
\medskip

\noindent {\bf Acknowledgements.}  The first-named author was supported by the La Trobe University Startup grant. The second-named author was supported by Schroedinger Scholarship J 3824 of the Austrian Science Fund (FWF). The third-named author was supported by the Royal Society Fellowship. Part of this work was carried out when the second- and third-named authors visited La Trobe University. We are thankful to La Trobe University and Australian Mathematical Sciences Institute (AMSI) for travel support. We would like to thank the anonymous referee for several useful comments and suggestions which has improved the clarity and presentation of the paper.

\section{Proof of the main theorem}

For completeness we give below a very brief introduction to Hausdorff measures and dimension. For further details see \cite{Falconer_book2013}.

\subsection{Preliminaries}\label{hm}

Let
$\Omega\subset \R^n$.
 Then for any $0 < \rho \leq \infty$, any finite or countable collection~$\{B_i\}$ of subsets of $\R^n$ such that
$\Omega\subset \bigcup_i B_i$ and $\mathrm{diam} (B_i)\le \rho$ is called a \emph{$\rho$-cover} of $\Omega$.
Let
\[ 
\HH_{\rho}^{f}(\Omega)=\inf \sum_{i} f\left(\diam (B_i)\right),
\]
where the infimum is taken over all possible $\rho$-covers 
$\{B_i\}$ of $\Omega$. The \textit{$f$-dimensional Hausdorff measure of $\Omega$} is defined to be
\[
\HH^f(\Omega)=\lim_{\rho\to 0}\HH_\rho^f(\Omega).
\]
The map $\HH^f:\P(\R^n)\to [0,\infty]$ is a Borel measure. In the case that $f(r)=r^s \;\; (s\geq 0)$, the measure $\HH^f$ is denoted $\HH^s$ and is called \emph{$s$-dimensional Hausdorff measure}. For any set $\Omega\subset \R^n$ one can easily verify that there exists a unique critical value of $s$ at which the function $s\mapsto\HH^s(\Omega)$ ``jumps'' from infinity to zero. The value taken by $s$ at this discontinuity is referred to as the \textit{Hausdorff dimension} of $\Omega$ and is denoted by $\dim_{\HH} \Omega $; i.e.
\[
\dim_\HH \Omega :=\inf\{s\geq 0\;:\; \HH^s(\Omega)=0\}.
\]
A countable collection $\{B_i\}$ is called a \emph{fine cover} of $\Omega$ if for every $\rho>0$ it contains a sub-collection that is a $\rho$-cover of $\Omega$.

\medskip

We state the Hausdorff measure analogue of the famous Borel--Cantelli lemma (see \cite[Lemma 3.10]{BernikDodson}) which will allow us to estimate the Hausdorff measure of certain sets via calculating the Hausdorff $f$-sum of a fine cover.

\begin{lemma}[Hausdorff--Cantelli lemma]\label{bclem}
Let $\{B_i\}\subset\R^n$ be a fine cover of a set $\Omega$ and let $f$ be a dimension function such that
\begin{equation}
\label{fdimcost}
\sum_i f\left(\diam(B_i)\right) \, < \, \infty.
\end{equation}
Then $$\HH^f(\Omega)=0.$$
\end{lemma}
In what follows we will call the series \eqref{fdimcost} the \emph{$f$-dimensional cost} of the collection $\{B_i\}$. Note that if $\{B_i\}$ is only a cover and not a fine cover, then there is no necessary relation between $f$-dimensional cost and $f$-dimensional Hausdorff measure.
Finally, we recall a very useful property of Hausdorff measures
under Lipschitz maps. For $\HH^s$ measure this property follows from \cite[Proposition 3.1]{Falconer_book2013}, but we give a different proof for $\HH^f$ measure.

\begin{proposition} \label{pro}
Fix positive integers $m,n$, a set $A\subset \mathbb{R}^{m}$, a Lipschitz continuous function $g: \R^{m}\to \R^n$, and a dimension function $f$. Then for $B=g(A)\subset \R^n$ we 
have $\HH^f(B)\lesssim \HH^f(A)$. If $g$ is bi-Lipschitz, then $\HH^f(B)\asymp\HH^f(A)$.
\end{proposition}

\begin{proof}
Let $L$ be a Lipschitz constant for $g$. Since we work in $\mathbb{R}^n$, there exists a constant $N$ such that for all $\rho > 0$, every set of diameter $L \rho$ can be covered by at most $N$ sets of diameter at most $\rho$. Now if $\{B_i : i\in\N\}$ is a $\rho$-cover of $A$, then the collection $$\{B_{i,j} : i\in\N,\; j = 1,\ldots,N\}$$ is a $\rho$-cover of $B$, where for each $i$, $$\{B_{i,j} : j = 1,\ldots,N\}$$ is a cover of $g(B_i)$ by sets of diameter $\leq\diam(B_i)$. Comparing the costs of these covers, taking the infimum over all such covers, and finally taking the limit as $\rho\to 0$, one sees that $\HH^f(B) \leq N \HH^f(A)$.
\end{proof}

\subsection{Proof of Theorem \ref{thm1}}

Recall that  $U\subset\R^{n-1}$ is a connected bounded open set and that $$S_\MM \df \{\xx\in U : \nabla^2 g(\xx) \text{ is singular}\}.$$
First note that $U\setminus S_\MM$ can be covered by countably many small balls $B_i$ such that $\nabla^2 g(\xx) \in F_i$ for all $\xx\in B_i$, where $F_i$ is a compact convex set of matrices with nonzero determinant. If $\HH^{f}(\WWW \cap \Gamma(g\given B_i)) = 0$ for all such $B_i$, then it follows from \eqref{Hessian} and the $\sigma$-subadditivity of measures that $$\HH^{f}(\WWW\cap \Gamma(g)) = 0.$$ Here $\Gamma(g\given B_i)$ is the graph of $g$ restricted to $B_i$, or equivalently $$\Gamma(g \given B_i)  =  (B_i \times \R ) \cap \Gamma(g).$$ Thus, we can without loss of generality assume that $U$ is a ball and that $\nabla^2 g(\xx)\in F$ for all $\xx\in U$, where $F$ is a compact convex set of matrices with nonzero determinant.

For any $p\in\Z$ and $\qq\in \Z^n$, we will bound the size of the set 
\[
S(p,\qq) = S_{\Psi,\theta}(p,\qq) = \{\xx\in K : |\qq\cdot (\xx,g(\xx)) - p - \theta(\xx)| < \Psi(\qq)\},
\]
where $K$ is a compact subset of $U$. First suppose that $q_n\neq 0$. Write $\qq = q_n(\rr,1)$ for some $\rr\in\Q^{n-1}$, and let $\rho = \Psi(\qq)/|q_n|$. Then
\[
S(p,\qq) = \big\{\xx\in K : \big|\rr\cdot\xx + g(\xx) - \tfrac{p+\theta(\xx)}{q_n}\big| < \rho\big\}.
\]
For fixed $p$ and $\qq$, define a function $h = h_{p,\qq}: \R^{n-1}\to \R$ by 
$$h(\xx) = \rr\cdot\xx + g(\xx) - \frac{p+\theta(\xx)}{q_n},$$ for
$\rr$ induced by $\qq$ as above. Finally, if $q_n = 0$, then we instead let
$$
\rr = (q_1,\ldots,q_{n-1}),\qquad
h(\xx) = \rr\cdot\xx - p - \theta(\xx),\qquad
\rho = \Psi(\qq).$$
Note that either way we have
\[
S(p,\qq) = \{\xx\in K : |h(\xx)| < \rho\}.
\]
\begin{claim}
\label{claim1}
For all but finitely many $\qq$, either there exists $\vv\in\R^{n-1}$ such that $$\|\nabla h(\xx)\| \asymp \|\xx-\vv\|\quad \text {for all }\quad \xx\in U, $$ or $$\|\nabla h(\xx)\| \asymp \|(\rr,1)\|\quad \text{for all}\quad \xx\in K.$$ Either way, we have $\|\nabla^2 h(\xx) \|\lesssim 1$ for all $\xx\in U$.
\end{claim}
\begin{proof}
Fix $\epsilon > 0$ such that the $\epsilon$-neighborhood of $F$, which we will denote by $\NN(F,\epsilon)$, consists of matrices whose determinant has absolute value $\geq \epsilon$. Assume first that $|q_n|$ is large enough so that $$\|\nabla^2 \theta(\xx)\| \leq \epsilon |q_n| \quad \text{for all} \quad \xx\in U.$$ Then for all $\xx\in U$, we have $\nabla^2 h(\xx) \in \NN(F,\epsilon)$. Now, for all $\xx,\yy\in U$ we have $$\nabla h(\yy) - \nabla h(\xx) = A(\yy - \xx),$$ where $$A = \int_0^1 \nabla^2 h(\xx + t (\yy - \xx)) \dee t.$$  Note that $A\in \NN(F,\epsilon)$ by the convexity of $F$. It follows that $$\|\nabla h(\yy) - \nabla h(\xx)\| \asymp \|\yy - \xx\|.$$ 
If $\nabla h(\vv) = \0$ for some $\vv\in U$, we are done, so suppose that $\nabla h(\xx) \neq \0$ for all $\xx\in U$. Then $\|\nabla h(\xx)\| \gtrsim 1$ for all $\xx\in K$. Fix such an $\xx$. Since $\|\nabla g(\xx)\|,\|\nabla \theta(\xx)\| \lesssim 1$, and $$\nabla h(\xx) = \rr + \nabla g(\xx) - \frac1{q_n} \nabla \theta(\xx),$$ it follows that $\|\nabla h(\xx)\| \asymp \|(\rr,1)\|$.

On the other hand, if $\|\rr\|$ is sufficiently large, then for all $\xx\in U$, the equation
\begin{equation}
\label{nablah}
\nabla h(\xx) = \begin{cases}
\rr + \nabla g(\xx) - \frac1{q_n} \nabla \theta(\xx) & \text{if $q_n \neq 0$}\\
\rr - \nabla\theta(\xx) & \text{if $q_n = 0$}
\end{cases}
\end{equation}
automatically implies that $\|\nabla h(\xx)\| \asymp \|(\rr,1)\|$. So the only way the claim can fail to hold is if $|q_n|$ and $\|\rr\|$ are both bounded, and there are only finitely many $\qq\in\Z^n$ such that this is true.

Finally, since $\nabla^2 g(\xx)$ and $\nabla^2\theta(\xx)$ are both bounded independent of $\xx$, so is
\[
\nabla^2 h(\xx) = \begin{cases}
\nabla^2 g(\xx) - \frac1{q_n}\nabla^2 \theta(\xx) & \text{if $q_n \neq 0$}\\
-\nabla^2 \theta(\xx) & \text{if $q_n = 0$}.
\end{cases}
\qedhere\]
\end{proof}

We now split the proof according to which of the two cases in Claim \ref{claim1} holds:\\

{\bf Case 1:} The point $\vv$ described in the claim exists. Note that by \eqref{nablah}, this implies that $\|\rr\| \lesssim 1$ and thus that $\max(|q_n|,1) \asymp \|\qq\|$. 

\begin{lemma}
\label{lemmacover}
Let $\phi: U \subset \R^{n-1}\to \R$ be a $C^2$ function. Fix $\alpha>0$, $\delta>0$, and $\xx\in U$ such that $B_{n-1}(\xx,\alpha) \subset U$.
There exists a constant $C > 0$ depending only on $n$ such that if
\begin{equation}\label{ine}
\|\nabla \phi(\xx)\| \geq C \alpha \sup_{\zz\in U} 
\|\nabla^2 \phi(\zz)\|,
\end{equation}
then the set $$S(\phi,\delta) = \{\yy\in B_{n-1}(\xx,\alpha) : |\phi(\yy)| < \|\nabla \phi(\xx)\| \delta\}$$ can be covered by $\asymp (\alpha/\delta)^{n-2}$ balls of radius $\delta$.
\end{lemma}

\begin{proof}
	Without loss of generality we can assume $\xx=\0\in U$ since translations do not affect the claim. 
For simplicity and the ease of notation let $\kappa:=\|\nabla \phi(\0)\|$.  Clearly $\kappa>0$ as otherwise
by assumption of the lemma
$\nabla^{2}\phi$ vanishes on $U$, that is, the set $S(\phi, \delta)$ is empty and hence the statement of the lemma is automatically satisfied. 
By rotating the coordinate system we assume that $\nabla \phi(\0)=\kappa\ee_{n-1}$, where $\ee_{n-1}$ denotes the $(n-1)$-st canonical base vector.
 Now consider the map 
 $$\Phi:  B_{n-1}(\0,\alpha)\to \R^{n-1}$$ 
 defined by the formula $$\Phi(\yy) = (\kappa y_1,\ldots,\kappa y_{n-3},\kappa y_{n-2},\phi(\yy)),$$
 for $\yy=(y_{1},\ldots,y_{n-1})$. We have $$\nabla\Phi(\0) = R:= \kappa I_{n-1}$$ where $I_{n-1}$ denotes the 
$(n-1)\times (n-1)$ identity matrix.

On the other hand
\[
\sup_{\zz\in B_{n-1}(\0,\alpha)} \|\nabla^2\Phi(\zz)\| = \sup_{\zz\in B_{n-1}(\0,\alpha)} \|\nabla^2 \phi(\zz)\| \leq \frac{\|\nabla\Phi(\0)\|}{C\alpha}= \frac{\kappa}{C\alpha}.
\]
Denote by $B(R,\kappa/(C\alpha))$ to be the set of $(n-1)\times (n-1)$ matrices whose entries differ by at most $\kappa/(C\alpha)$ from $R$ in terms of the Euclidean norm
on $\mathbb{R}^{(n-1)^{2}}$. Since $B_{n-1}(\0,\alpha)$ has diameter
$2\alpha$,
it follows from the mean value inequality applied to the gradient $\nabla\Phi$
that $$\nabla\Phi(\zz) \in B(R,2\alpha\cdot \kappa/(C\alpha))=B(R,2\kappa/C)\quad \text{for all}\quad  \zz\in B_{n-1}(\0,\alpha).$$ Thus by a further usage of the mean value inequality
applied to $\Phi$, for all $\yy,\zz\in B_{n-1}(\0,\alpha)$ there exists $A_{\yy,\zz}\in B(R,2\kappa/C)$ depending
	on $\yy,\zz$, such that $$\Phi(\zz) - \Phi(\yy) = A_{\yy,\zz}(\zz-\yy).$$ 
	Since the set of invertible $(n-1)\times (n-1)$ matrices
forms an open subset of $\mathbb{R}^{(n-1)\times (n-1)}$ and 
$\kappa>0$, so it follows from the Inverse Function Theorem
that if $C$ is sufficiently large, then $B(R,2\kappa/C)$ consists of
invertible matrices. Hence the map 
$\Phi$ is a bi-Lipschitz map with a uniform bi-Lipschitz constant. Notice that the set $S(\phi,\delta) $ can now be rewritten as
\[
S(\phi,\delta) = \Phi^{-1}\big((\0,\alpha)^{n-2}\times (-\delta\kappa,\delta\kappa)\big)
\]
and it is clear that this set can be covered 
by $\asymp \kappa (\delta/\alpha)^{-(n-2)}=\kappa(\alpha/\delta)^{n-2}$ balls 
of radius $\delta$, with the implied constant depending only on
the bi-Lipschitz constant. Since $\kappa$ is fixed the proof is finished.
\end{proof}

Now fix $k\in \Z$ and consider the annulus 
$$A_k = B_{n-1}(\vv,2^{-k}) \setminus B_{n-1}(\vv,2^{-(k+1)}).$$ 
Observe that the family $(A_k)_{k\in\Z}$ forms a partition of $\R^{n-1}\setminus \{\vv\}$,
and when restricting to $K$ it suffices to take into account
indices $k\geq k_{0}$ for some absolute constant $k_0$.
For all $\xx\in A_k\cap K$, 
we have

\[
\|\nabla h(\xx)\| \asymp \|\xx - \vv\| \asymp 2^{-k}
\]
and thus if $\epsilon > 0$ is sufficiently small (depending on the constant $C$ appearing in Lemma \ref{lemmacover}),  then we have
\[
\|\nabla h(\xx)\| \geq C (\epsilon 2^{-k}) \sup_{\zz\in K} \|\nabla^2 h(\zz)\|.
\]
Where we have used the claim \ref{claim1} that  $\|\nabla^2 h(\zz) \|\lesssim 1$ for all $\zz\in U$. Letting $\alpha = \epsilon 2^{-k}$, $\delta = \rho/\|\nabla h(\xx)\| \asymp 2^k \rho$, and $\phi=h$, from Lemma \ref{lemmacover}
we see that $$S(h,\delta) = S(p,\qq)\cap B_{n-1}(\xx,\epsilon 2^{-k})$$ 
can be covered by $\asymp_{\epsilon} (2^{-2k}/\rho)^{n-2}$ 
balls of radius $\asymp 2^k \rho$. Thus,
\[
\HH^f_\infty \big(S(p,\qq)\cap B_{n-1}(\xx,\epsilon 2^{-k})\big) \lesssim \left(\frac{2^{-2k}}{\rho}\right)^{n-2} f\big(2^k \rho\big).
\]
Now, there exists a universal constant $N$ depending on $n$ such that $A_k$ can be covered by $N$ balls of radius $\epsilon 2^{-k}$ centred in $A_k$. Since the above estimate holds
for any $\xx\in A_{k}$, multiplying
the implied constant above by $N$ we obtain 
\[
\HH^f_\infty \big(S(p,\qq)\cap A_k\big) \lesssim \left(\frac{2^{-2k}}{\rho}\right)^{n-2} f\big(2^k \rho\big),
\]
and summing over $k$ gives
\[
\HH^f_\infty \big(S(p,\qq)\big) \lesssim \sum_{k\geq k_0} \left(\frac{2^{-2k}}{\rho}\right)^{n-2} f\big(2^k \rho\big),
\]
where $k_0\in\Z$ is the smallest integer such that $A_{k_0}\cap K\neq \emptyset$. Now, by \eqref{fhypothesis}, we have
\[
f\big(2^k \rho\big) \lesssim 2^{sk} f(\rho)
\]
and thus
\begin{align*}
\HH^f_\infty \big(S(p,\qq)\big)
&\lesssim \rho^{-(n-2)} f(\rho) \sum_{k\geq k_0} 2^{-k[2(n-2)-s]} \noreason\\
&\asymp \rho^{-(n-2)} f(\rho) \since{$s < 2(n-2)$} \\
& = F(\rho) \note{where $F(x) = x^{-(n-2)} f(x)$}\\
& = F\left(\frac{\Psi(\qq)}{\max(|q_n|,1)}\right)\\ &\asymp F\left(\frac{\Psi(\qq)}{\|\qq\|}\right). \noreason
\end{align*}

{\bf Case 2:} No such point $\vv$ exists, and thus $\|\nabla h(\xx)\| \asymp \|(\rr,1)\|$ for all $\xx\in K$. If $\epsilon > 0$ is sufficiently small, then
\[
\|\nabla h(\xx)\| \geq C \epsilon \sup\|\nabla^2 h\|
\]
for all $\xx\in K$. Applying Lemma \ref{lemmacover} with $\alpha = \epsilon \asymp_\epsilon 1$, $$\delta = \rho/\|\nabla h(\xx)\| \asymp \rho/\|(\rr,1)\|,$$ and $\phi=h$, we see that 
$S(p,\qq) \cap B_{n-1}(\xx,\epsilon)$
 can be covered by $\asymp_\epsilon (\|(\rr,1)\|/\rho)^{n-2}$ balls of radius $\asymp \rho/\|(\rr,1)\|$. So
\begin{align*}
\HH^f_\infty \big(S(p,\qq)\big) &\lesssim \left(\frac{\|(\rr,1)\|}{\rho}\right)^{n-2}  f\left(\frac{\rho}{\|(\rr,1)\|}\right) \\ & = F\left(\frac{\Psi(\qq)}{\max(|q_n|,1)\cdot\|(\rr,1)\|}\right) \\ &\asymp F\left(\frac{\Psi(\qq)}{\|\qq\|}\right).
\end{align*}
{\bf Conclusion.} Combining the two cases, we have shown that for all but finitely many $\qq \in \Z^n$, we have \[
\HH^f_\infty \big(S(p,\qq)\big) \lesssim F\left(\frac{\Psi(\qq)}{\|\qq\|}\right) \text{ for all } p\in\Z.
\]
Now since
\[
E_K = \{\xx\in K : (\xx,g(\xx)) \in \WWW\} = \limsup_{\qq\to \infty} \bigcup_{p\in\Z} S(p,\qq),
\] 
we can make the transition from 
$\HH^f_\infty$ to $\HH^{f}$ by using 
the Hausdorff--Cantelli lemma \ref{bclem}:
\[
\HH^{f}(E_K) = 0 \qquad \text{ if } \qquad \sum_{\qq\in \Z^n\setminus \{\0\}} 
\sum_{\substack{p\in \Z \\ S(p,\qq) \neq \smallemptyset}} F\left(\frac{\Psi(\qq)}{\|\qq\|}\right) < \infty.
\]
Now for each $\qq\in \Z^n$ 
the number of $p\in\Z$ such that $S(p,\qq) \neq \emptyset$ 
is $\lesssim \|\qq\|$ by the compactness of $K$, 
so the above series can be estimated by
\begin{align*}
\sum_{\qq\in \Z^n\setminus \{\0\}} \sum_{\substack{p\in \Z \\ S(p,\qq) \neq \smallemptyset}} F\left(\frac{\Psi(\qq)}{\|\qq\|}\right) & \lesssim \sum_{\qq\in \Z^n\setminus \{\0\}} \|\qq\| F\left(\frac{\Psi(\qq)}{\|\qq\|}\right)
\\ &=\sum\limits_{\qq\in\Z^n\setminus \{\0\}}\|\qq\|^{n-1}\Psi(\qq)^{2-n}f\left(\frac{\Psi(\qq)}{\|\qq\|}\right).
\end{align*}
We recognise the right hand side sum
as the series \eqref{eqcon}, which converges by assumption. 
So $\HH^f(E_K)=0$, 
and since the map $G(\xx) = (\xx,g(\xx))$ is Lipschitz on $K$, it follows from Proposition~\ref{pro}
that $$\HH^f(G(E_K)) = \HH^f(\WWW\cap \Gamma(g\given K)) = 0.$$ Since $K\subset U$ was arbitrary, this implies \eqref{fw}.

\section{Singular Hessian condition}\label{sectionfibering}

In this section we characterise the Hessian condition on the function $g$ i.e. the condition (II). First we discuss this condition on $g$ in form of three examples. We recall that $g:U\to\R$ is a $C^2$ function, where $U\subset\R^{n-1}$ is a connected bounded open set, that $\MM$ is the graph of $g$, and that
\[
S_\MM = \{\xx\in U : \nabla^2 g(\xx) \text{ is singular}\}.
\]
We want to study the question of when
\begin{equation}
\label{Hessianv2}
\HH^f(S_\MM) = 0,
\end{equation}
where $f$ is a dimension function.

\begin{example}
Suppose that $g$ is an analytic function. Then $S_\MM$ is an analytic set, so if $S_\MM \neq U$, then $\HD(S_\MM) \leq \dim(U) - 1 = n-2$. In this case, \eqref{Hessianv2} holds for all dimension functions $f$ such that $$\liminf_{r\to 0} \frac{\log f(r)}{\log r}>n-2.$$ So the main obstacle to \eqref{Hessianv2} is when $S_\MM = U$, or equivalently when $\nabla^2 g(\xx)$ is singular for all $\xx\in U$. We will show  below that this can only occur when a dense open subset of the manifold $\MM$ can be fibered into (i.e. written as the disjoint union of) line segments.

On the other hand, from Corollary \ref{cor2}, it follows that $$\dim_\HH (\WWW\cap\MM)>n-2$$ and hence in the case of $$\liminf_{r\to 0} \frac{\log f(r)}{\log r}\leq n-2$$ we always have $$\HH^f\left((\WWW\cap\MM\right)=\infty$$ for any dimension function $f$.
\end{example}
\begin{example}
\label{examplealgebraic}
Suppose that $g: U = \R^{n-1}\to \R$ is a multivariate real polynomial of total degree at most $N\geq 2$, i.e.
\[
g(\xx)=\sum_{j=1}^{J} b_{j}x_{1}^{a_{1,j}}x_{2}^{a_{2,j}}\cdots 
x_{n-1}^{a_{n-1,j}}
\]
where $(a_{1,j},\ldots,a_{n-1,j})$ run through all tuples of integers $a_{i,j}\geq 0$ whose sum over $1\leq i\leq n-1$ for each $j$ does not exceed $N$, $J$ denotes the number of such tuples, and $b_{j}\in\R$ for $1\leq j\leq J$. Then as in the previous example, unless the determinant of the Hessian vanishes identically on $\mathbb{R}^{n-1}$, condition \eqref{Hessianv2} is satisfied for dimension functions $f$ such that $$\liminf_{r\to 0} \log f(r)/\log r>n-2,$$ since the exceptional set is a finite union of hypersurfaces. It is thus easy to see that \eqref{Hessianv2} holds for Lebesgue almost all choices of $\underline{b}=(b_{1},\ldots,b_{J})\in\mathbb{R}^{J}$. On the other hand, typically we expect the variety $S_\MM$ to have dimension precisely $n-2$, and thus if $f(t)=t^{s}$ for $s\leq n-2$ then the condition \eqref{Hessianv2} will fail generically. Also notice that the Lebesgue nullset of $\underline{b}=(b_{1},\ldots,b_{J})\in\mathbb{R}^{J}$ where the determinant of the Hessian vanishes identically is nonempty; for example, it contains those $\underline b$ such that the corresponding $g(\xx)$ is independent of one of the coordinates of $\xx$. When $n=2$ and $n=3$, another example is given by
\begin{equation}
\label{b1b6}
g(x_{1},x_{2})=b_{1}x_{1}^{2}+b_{2}x_{1}x_{2}+b_{3}x_{2}^{2}+b_{4}x_{1}+b_{5}x_{2}+b_{6} \quad \text{with}\quad b_{2}^{2}-4b_{1}b_{3}=0.
\end{equation}
 This example can be generalised to any $n\geq 3$.
The vanishing Hessian
condition has been intensely studied
for {\em homogeneous} polynomials $g$. For $n\leq 5$ in our
notation, in this setting the condition implies that after an appropriate change of variables, the function $g$ is independent of at least one of the variables---this was proven by Gordan and Noether \cite{GordanNoether}, who also constructed a counterexample when $n=6$, thus disproving a conjecture of Hesse:
\begin{equation}
\label{6dimex}
g(x_{1},x_{2},x_{3},x_{4},x_{5})=x_{1}^{2}x_{3}+x_{1}x_{2}x_{4}+x_{2}^{2}x_{5}+x_{1}^{3}+x_{2}^{3}.
\end{equation}
See also \cite{Lossen} for more on vanishing Hessian.  Note that if we let $b_4=b_5=b_6=0$ in \eqref{b1b6}, then we have $g(x_{1},x_{2})=(\sqrt{b_{1}}x_{1} \pm \sqrt{b_{3}}x_{2})^{2}$ and thus Gordan and Noether's theorem can be demonstrated by the change of variables $y_1 = \sqrt{b_{1}}x_{1} \pm \sqrt{b_{3}}x_{2}$, $y_2 = \sqrt b_2 x_1 \mp \sqrt b_1 x_2$. Correspondingly, the graph of $g$ can be fibered into lines of the form $y_1 = \text{constant}$. Also note that if $g$ is given by \eqref{6dimex}, then $\MM$ is fibered into three-dimensional planes of the form $$\MM \cap \{x_1=a_1,x_2=a_2\}.$$ However, the proof of Theorem \ref{theoremfibering} below only yields a fibering into lines:
\[
\{(\xx,g(\xx)) + t (0,0,x_2^2,-2x_1 x_2,x_1^2,0) : t\in\R\}.
\]
\end{example}
\begin{example} In the case $n=2$ 
where $g: \mathbb{R}\to\mathbb{R}$ 
and $S_\MM=\{ x\in\mathbb{R}: g^{\prime\prime}(x)=0\}$, there
exist
$C^{2}$ functions $g$ for which the induced set $S_\MM$ has
positive one-dimensional Lebesgue measure but empty interior. To see this,
recall that any closed set $T$ 
can be realised as the vanishing set of some continuous
function $h$ by letting $h(x)=d(x,T)$ to be the distance from $x$ to $T$. Taking a fat Cantor set~\cite[pp.140--141]{AliprantisBurkinshaw}
for $T$ and integrating the continuous function 
$h(x)=d(x,T)$ twice, 
i.e. $$r(x)= \int_{t=0}^{x} h(t) dt \quad \text{ and}\quad g(x)= \int_{t=0}^{x} r(t) dt, $$ one checks that
we obtain a suitable function $g$. 
\end{example}

\begin{theorem}
\label{theoremfibering}
Let $U$ be a connected bounded open subset of $\R^{n-1}$. Let $g:U \to \R$ be a $C^2$ function whose Hessian is singular at every point. Then there is a dense open subset $V\subset U$ such that the manifold $\MM = \Gamma(g\given V) = G(V)$ can be fibered into line segments, where $G(\xx) = (\xx,g(\xx))$.
\end{theorem}

\begin{proof}
Let $\pp \in U$ be a point such that the function
\[
\xx \mapsto \dim\Ker\nabla^2 g(\xx)
\]
is constant in a neighborhood of $\pp$. Note that since this function is integer-valued and upper semicontinuous, the set of such points $\pp$ forms a dense open set in $U$. Also note that by our assumption on $g$, the constant value that this function takes on in a neighborhood of $\pp$ is greater than or equal to $1$. Now let $\NN \subset U$ be a manifold which passes through $\pp$ orthogonally to $\Ker\nabla^2 g(\pp)$, for example the affine subspace passing through $\pp$ and orthogonal to $\Ker\nabla^2 g(\pp)$.  Note that if $\NN$ is not an affine subspace, then the tangent space at $\pp$ is orthogonal to
$\Ker\nabla^2 g(\pp)$. Then the Jacobian of the transformation $$\nabla^2 g(\pp) : T_{\pp} \NN \to \R^{n-1}$$ is injective, so by shrinking $\NN$ if necessary, we can assume that the map $\nabla g:\NN\to \R^{n-1}$ is an embedding. Note that by shrinking of $\NN$ we mean taking its intersection with a small neighbourhood of $\pp$ (not a homothety). On the other hand, if $\gamma:I \subset\R \to U$ is any path such that $$\gamma'(t) \in \Ker\nabla^2 g(\gamma(t)) \text { for all } \ t,$$ then $\nabla g$ is constant on the image of $\gamma$. Let $W$ be a neighbourhood of $\pp$ such that every point in $W$ can be connected to a point of $\NN$ by such a path; then the image of $W$ under $\nabla g$ is contained in the image of $\NN$.

Since $\nabla g :\NN \to \R^{n-1}$ is an embedding, $\nabla g(\NN)$ is a smooth manifold. Thus, there exist a neighbourhood $X$ of the point $\qq = \nabla g(\pp)$ and a smooth function $h:X\to \R$ such that $h(\nabla g(\xx)) = 0$ for all $\xx\in \NN$ (and thus all $\xx\in W$), but $\nabla h(\yy) \neq \0$ for all $\yy\in X$. By shrinking $W$ we can assume that $$\nabla g(\xx) \in X \text{ for all } \xx\in W.$$

Taking the derivative of the equation $h(\nabla g(\xx)) = 0$ with respect to the $i$th coordinate gives 
\[
0 = h'(\nabla g(\xx)) [ \del_i \nabla g(\xx) ] = \big(\nabla^2 g(\xx) \cdot \nabla h(\nabla g(\xx))\big)_i.
\]
Since $i$ was arbitrary, this tells us that $$\vv(\xx) := \nabla h(\nabla g(\xx)) \in \Ker \nabla^2 g(\xx).$$ As before, if $\gamma : I \subset \R \to W$ is any curve such that $\gamma'(t) = \vv(\gamma(t))$ for all $t$, then $\nabla g$ is constant on the image of $\gamma$. (Note that since $\vv$ is continuously differentiable, for each $\xx\in W$ there exists such a $\gamma$ such that $\gamma(0) = \xx$.) But then $\vv$ is also constant on the image of $\gamma$, and so $\gamma'$ is constant. Thus $\gamma$ is linear, and since $\nabla g$ is constant on the image of $\gamma$, $g\circ \gamma$ is also linear. Thus, the image of $g\circ\gamma$ is a line segment in the manifold $\MM$, and since $\gamma$ was arbitrary (relative to the condition $\gamma'(t) = \vv(\gamma(t))$), this shows that $G(W)$ can be fibered into such line segments. Finally, since $\pp$ was arbitrary we can take the union of the corresponding sets $G(W)$ to get a dense open subset of $\MM$ that can be fibered into line segments.
\end{proof}

\providecommand{\bysame}{\leavevmode\hbox to3em{\hrulefill}\thinspace}
\providecommand{\MR}{\relax\ifhmode\unskip\space\fi MR }
\providecommand{\MRhref}[2]{%
  \href{http://www.ams.org/mathscinet-getitem?mr=#1}{#2}
}
\providecommand{\href}[2]{#2}

\end{document}